\documentclass[a4paper,11pt,fleqn]{article}

\usepackage{amsmath,amssymb,amsthm}
\usepackage[mathscr]{eucal}
\usepackage{graphicx}
\usepackage{xcolor}
\usepackage{eufrak,mathrsfs,longtable}

\newcommand{\todo}[1][\null]{\ensuremath{\clubsuit}}
\makeatletter
\long\def\@makecaption#1#2{%
  \vskip\abovecaptionskip\footnotesize
  \sbox\@tempboxa{#1. #2}%
  \ifdim \wd\@tempboxa >\hsize
    #1. #2\par
  \else
    \global \@minipagefalse
    \hb@xt@\hsize{\hfil\box\@tempboxa\hfil}%
  \fi
  \vskip\belowcaptionskip}
\makeatother

\newtheorem{theorem}{Theorem}

\newtheorem*{problem*}{Problem}
{\theoremstyle{definition}

\newtheorem*{remark*}{Remark}
}

\begin{document}

\par\noindent {\LARGE\bf
On the invariance properties of Vaidya-Bonner geodesics via symmetry operators\par}

{\vspace{5mm}\par\noindent {\large Davood Farrokhi$^{\dag}$, Rohollah Bakhshandeh-Chamazkoti$^\S$ and Mehdi Nadjafikhah$^{\ddag}$
} \par\vspace{3mm}\par}

{\vspace{3mm}\par\noindent {\it
$^{\dag}$Department of Mathematics, Karaj Branch,
Islamic Azad University, karaj, Iran.
}}

{\vspace{3mm}\par\noindent {\it
$^\S$Department of Mathematics,  Faculty of Basic Sciences,
Babol  Noshirvani University of Technology, Babol, Iran.\par

}}
{\vspace{3mm}\par\noindent {\it
$^\ddag$School of Mathematics,
Iran University of Science and Technology
Narmak, Tehran, 16846-13114, Iran
}}

{\vspace{3mm}\par\noindent {\it
\textup{E-mail:} davoodfarokhi61@gmail.com, r\_bakhshandeh@nit.ac.ir, m\_nadjafikhah@iust.ac.ir
}\par}

\vspace{8mm}\par\noindent\hspace*{10mm}\parbox{140mm}{\small
In the present paper, we try to investigate the Noether symmetries and Lie point symmetries of the
Vaidya-Bonner geodesics. Classification of one-dimensional subalgebras of Lie
point symmetries are considered. In fact, the collection of pairwise non-conjugate one-dimensional
subalgebras that are called the optimal system of subalgebras is determined. Moreover, as illustrative
examples, the symmetry analysis is implemented on two special cases of the system.}\par\vspace{4mm}

\section{Introduction}

Noether's theorem \cite{Noe1, Olv1} provides a method for finding conservation laws of differential
equations arising from a known Lagrangian and having a known Lie symmetry.
This theorem relies on the availability of a Lagrangian and the
corresponding Noether symmetries which leave invariant the action integral, \cite{Mahomed,Narain1}.

Since the geodesic equations follow from the variation of the geodesic
Lagrangian defined by the metric and due to the fact that the Noether symmetries
are a subgroup of the Lie symmetries of these equations, one should expect
a relation of the Noether symmetries of this Lagrangian with the projective
collineations of the metric or with its degenerates, \cite{Bakhshandeh1, Bakhshandeh2, Ibragimov, Tsamparlis}.
In some works, \cite{Bokhari, Kara}, the relation between  the Noether symmetries
and the Lie point symmetries (Killing vectors) of some special spacetimes are discussed.

In \cite{Tsamparlis},  Tsamparlis and Paliathanasis have computed the Lie point symmetries and the Noether symmetries explicitly
together with the corresponding linear and quadratic first integrals
 for the Schwarzschild spacetime and the Friedman Robertson Walker (FRW)
spacetime. These authors, in another paper \cite{Paliathanasis},  have proved a theorem that relates the Lie
symmetries of the geodesic equations in a Riemannian
space with the collineations of the metric. They applied
the results to Einstein spaces and spaces of constant
curvature.

In, \cite{Narain1, Narain2},  the authors present a complete analysis of symmetries of classes of wave equations that arise as a consequence
of some Vaidya metrics.
Now in this paper, we try to find Lie point symmetries and Noether symmetries for
the Vaidya-Bonner metric
\begin{eqnarray}
ds^{2}=-\left(1-\frac{M(t)}{r}+\frac{Q(t)}{r^{2}}\right)dt^{2}-2dtdr
+r^{2}(d\theta^2+\sin^{2}\theta \,d\phi^2)\,, \label{eq:VB}
\end{eqnarray}
where $M(t)$ is dynamical mass of the black hole and $Q(t)$ is electric charges that both of them depends to  the advanced Eddington time coordinate $t$, \cite{Niu}.
As examples we consider two cases with
$M(t)=1, Q(t)=t$ and $M(t)=t, Q(t)=t^2$.

In general relativity, the Vaidya metric describes the non-empty external spacetime of a
spherically symmetric and nonrotating star which is either emitting or absorbing null dust.
It is named after the Indian physicist Prahalad Chunnilal Vaidya and constitutes the simplest non-static
generalization of the non-radiative Schwarzschild solution to Einstein's field equation,
and therefore is also called the {\it radiating (shining) Schwarzschild metric}.
%
\section{Noether symmetries}
Suppose $(M, g)$ is Riemannian manifold  of dimension $n$. In a local space-time coordinate like that
${\bf x}(s)=(x^1(s),\ldots,x^n(s))$, the geodesic equations form a system of $n$
nonlinear, second order ordinary differential equations
\begin{eqnarray}\label{geodesicequation}
\ddot{x}^i(t)+\Gamma^i_{jk}\dot{x}_j(t)\dot{x}_k(t)=0,\qquad i=1, 2, \ldots, n,
\end{eqnarray}
where $\Gamma^i_{jk}$ are the Christoffel symbols and ``.'' represents derivative in terms of $s$.
Consider a second order system of $m$ ordinary differential equations (\ref{geodesicequation}) with following form
\begin{eqnarray}\label{eq:system}
 E_i(s, {\bf x}(s), {\bf x}^{(1)}(s), {\bf x}^{(2)}(s))=0, \quad i=1, \ldots, n.
\end{eqnarray}
Consider a vector field which is defined on a
real parameter fibre bundle over the manifold, \cite{Ibragimov},
\begin{eqnarray}
{\bf X}=\overline{\xi}(s, x^\mu)\partial_s+\overline{\eta}^\nu(s, x^\mu)\partial_ {x^\nu},
\end{eqnarray}
where $\mu, \nu= 1, 2, 3, 4$. The first prolongation of the above vector field defined on the real parameter fibre bundle over the tangent
bundle to the manifold, is expressed as follows:
\begin{eqnarray}\label{eq:firstprol}
{\bf X}^{[1]}={\bf X}+\left(\overline{\eta}^\nu_{,s}+\overline{\eta}^\nu_{,\mu}\dot{x}^\mu-\overline{\xi}_{,s}\dot{x}^\nu-\overline{\xi}_{,s}\dot{x}^\mu\dot{x}^\nu\right)\partial_{\dot{x}^\nu},
\end{eqnarray}
then ${\bf X}$ is a {\it Noether point symmetry} of the Lagrangian, \cite{Ibragimov},
\begin{eqnarray}
{\mathcal L}(s, x^\mu, \dot{x}^\mu)=g_{\mu\nu}(x^\sigma)\dot{x}^\mu\dot{x}^\nu,
\end{eqnarray}
if there exists a gauge function, $A(s, x^\mu)$, such that
\begin{eqnarray}\label{eq:tasa}
{\bf X}^{[1]}{\mathcal L}+(D_s\xi){\mathcal L}=D_sA,
\end{eqnarray}
where
\begin{eqnarray}\label{eq:ds}
D_s=\partial_ s+ \dot{x}^\mu\partial_{x^\mu}.
\end{eqnarray}
Since the corresponding Euler-Lagrange equations (geodesic equations) are second order ordinary differential equations,
 one generally takes first order Lagrangian.
 Particularly, we take ${\mathcal L}(s, x_i, \dot{x}_i)$, where ``dots'' denotes differentiation in terms of $s$,
 which results a set of second ODEs
\begin{eqnarray}\label{eq:secode}
\ddot{x}^\mu=g(s, x^\mu, \dot{x}^\mu).
\end{eqnarray}
\section{Lie point symmetries}
Consider a second order system of $m$ ordinary differential equations
(\ref{eq:system}) as geodesic equations of our given Riemannian metric $g$.
We consider a one-parameter Lie group of transformations acting on $(s, {\bf x})$-space with following forms
\begin{eqnarray}\label{eq:onepara}
\bar{s}= s+\epsilon\xi(s, {\bf x}(s))+O(\epsilon^2),\qquad \bar{x}^\alpha(s)= x^\alpha(s)+\epsilon\eta^\alpha(s, {\bf x}(s))+O(\epsilon^2),
\end{eqnarray}
where $\alpha=1, \ldots, n$. The infinitesimal generator ${\bf X}$ associated with the group of transformations (\ref{eq:onepara}) are
\begin{eqnarray}\label{infigener}
{\bf X}=\xi(s, {\bf x})\partial_s+\eta^\alpha(s, {\bf x})\partial_{x^\alpha},
\end{eqnarray}
the second order prolongation of ${\bf X}$ is given by
\begin{eqnarray}\label{jprolong}
{\bf X}^{[2]}={\bf X}+\eta^\alpha_{,(1)}(s, {\bf x}, {\bf x}^{(1)})\partial_{ x^\alpha_{,(1)}}+
\eta^\alpha_{,(2)}(s, {\bf x}, \ldots, {\bf x}^{(2)})\partial_{x^\alpha_{,(2)}},
\end{eqnarray}
in which the prolongation coefficients are
\begin{eqnarray}\label{jprolongcoeff}
\eta^\alpha_{,(1)}=D\eta^\alpha_{,(0)}-x^\alpha_{,(1)}D\xi,\qquad \eta^\alpha_{,(2)}=D\eta^\alpha_{,(1)}-x^\alpha_{,(2)}D\xi
\end{eqnarray}
where $\eta^\alpha_{,(0)}=\eta^\alpha(s, {\bf x})$ and $D$ is the total derivative operator.

The invariance of the system (\ref{eq:system})   under the one--parameter Lie group of transformations (\ref{eq:onepara})
leads to the invariance criterions.
So ${\bf X}$ is a point symmetry generator of (\ref{eq:system}) if and
only if
\begin{eqnarray}\label{invcond}
{\bf X}^{[2]}E_i\big|_{E_i=0}=0.
\end{eqnarray}
Using (\ref{invcond}) we find a system of  partial differential equations that is called the {\it determining equations}.
By solving the determining PDEs, the symmetry operators of the considered geodesics will be found.
\section{Symmetry computation of Vaidya-Bonner geodesics in general form}
\begin{theorem}
The Lie algebra of Noether symmetries
associated to Vaidya-Bonner metric (\ref{eq:VB}) for the arbitrary functions  $M(t)$ and $Q(t)$ is generated by
following infinitesimals:
\begin{gather}\label{Noethersymm1}
\begin{array}{lcl}
 &{\bf X}_1=\partial_s, \qquad {\bf X}_2=\partial_t,\qquad  X_3=\partial_\varphi \qquad
 {\bf X}_4=-\cos\varphi\;\partial_\theta+\sin\varphi\cot\theta\;\partial_\varphi,\\[3mm]
 &{\bf X}_5=\sin\varphi\;\partial_\theta+\cos\varphi\csc\theta\;\partial_\varphi.
\end{array}
\end{gather}
\end{theorem}
\begin{proof}
Associated  Lagrangian of the (\ref{eq:VB}) metric for arbitrary functions  $M(u)$ and $Q(u)$ is
\begin{eqnarray}
{\mathcal L}=-\left(1-\frac{M(t)}{r}+\frac{Q(t)}{r^{2}}\right)\dot{t}^{2}-2\dot{t}\dot{r}+r^{2}(\dot{\theta}^2+\sin^{2}\theta\,\dot{\phi}^2),\label{Lagrangian}
\end{eqnarray}
where the ``dot'' represents derivative in terms of $s$.
Suppose that the vector field
$${\bf X}=\xi\;\partial_s+\eta^1\;\partial_t+\eta^2\;\partial_r+\eta^3\;\partial_\theta
+\eta^4\;\partial_\varphi,$$
is the Noether symmetry operator of (\ref{Lagrangian}). Assuming $A=0$ in  (\ref{eq:tasa}) formula, we have
\begin{eqnarray}\label{eq:tasa1}
{\bf X}^{[1]}{\mathcal L}+(D_s\xi){\mathcal L}=0,
\end{eqnarray}
where
\begin{eqnarray}\label{eq:ds1}
D=\partial_s+\dot{t}\;\partial_t+\dot{r}\;\partial_r+\dot\theta\;\partial_\theta+\dot\varphi\;\partial_\varphi.
\end{eqnarray}
The equation (\ref{eq:tasa1}) leads to some differential equations that their solutions are the Noether symmetries (\ref{Noethersymm1}).
\end{proof}
\begin{theorem}
Lie algebra of point symmetries
associated to Vaidya-Bonner metric \eqref{eq:VB}, for  arbitrary functions $M(u)$ and $Q(u)$  is generated by:
\begin{gather} \label{Liepoint1}
\begin{array}{lcl}
&{\bf X}_1=\partial_s, \quad {\bf X}_2=\partial_t, \quad {\bf X}_3=\partial_\varphi, \quad {\bf X}_4=\cot(\theta)\sin(\varphi)\partial_\varphi-\cos(\varphi)\partial_\theta,\\[3mm]
&{\bf X}_5=\cot(\theta)\cos(\varphi)\partial_\varphi+\sin(\varphi)\partial_\theta.
\end{array}
\end{gather}
\end{theorem}
\begin{proof}
The Euler-Lagrange geodesic equations associated with the Lagrangian (\ref{Lagrangian}) are
\begin{eqnarray}\label{eq:eulerlagra}
\left\{ \begin{array}{lcl}
E_1:3\ddot{t}+\frac{2Q(t)-r M(t)}{4r^3}\dot{t}^2+\frac{r\sin^{2}\theta}{2}\dot{\phi}^2+\frac{r}{2}\ddot{\theta}^2=0, \\[2mm]
E_2:5\ddot{r}+\frac{-2(rM^{'}(t)+4Q^{'}(t))(r M(t)-2Q(t))(r M(t)-r^2-Q(t))}{8r^5}\dot{t}^2\\[2mm] \hspace*{7mm}+\frac{r M(t)-2Q(t)}{4r^3}\dot{t}\dot{r}+\frac{M(t)-r^2-Q(t)}{4r}\dot{\theta}^2+\frac{(M(t)r-r^2-Q(t))(sin^2\theta)}{4r}\ddot{\phi}^2=0, \\[2mm]
E_3:3\ddot{\theta}+\frac{2}{r}\dot{r}\dot{\theta}-\sin(\theta)\cos(\theta)\dot{\phi}^2=0,\\[2mm]
E_4:4\ddot{\phi}+\frac{1}{r}\dot{r}\dot{\phi}+2\cot(\theta)\dot{\theta}\dot{\phi}=0 \\
\end{array}\right.
\end{eqnarray}
where the $M'(t),Q'(t)$ represents derivative in terms of $t$.
Now the second prolongation of the operator ${\bf X}$ is computed. ${\bf X}$. Applying the invariance criteria (\ref{invcond}) on equations  (\ref{eq:eulerlagra})
we have some determining equations. By inserting \eqref{infigener} to invariance condition \eqref{invcond}
and then comparing of the powers $\dot{t},\dot{r},\dot{\theta},\dot{\phi}$ ,
we obtain complete set of determining equations. By solving this system of PDEs we find that
\begin{eqnarray*}
&\xi=c_1,   \qquad  \eta^1=c_2, \qquad \eta^2=0,
\eta^3=-c_4\cos(\varphi)+c_5\sin(\varphi),&\\
&\eta^4=c_3+\cot(\theta)(c_4\sin(\varphi)+c_5\cos(\varphi)).&
\end{eqnarray*}
\end{proof}
\section{Algebraic structure of the operators \eqref{Liepoint1}}
The algebra ${\mathfrak g}=\langle{\bf X}_1, {\bf X}_2,\cdots,{\bf X}_5\rangle$ is non-solvable because
we have ${\mathfrak g^{(1)}}=[{\mathfrak g} ,{\mathfrak g}]=\langle{\bf X}_2,{\bf X}_3,{\bf X}_4\rangle$ and
we have the following chain of ideals ${\mathfrak g}\supset {\mathfrak g^{(1)}}\supset {\mathfrak g^{(2)}}\dots \neq 0$
which shows that ${\mathfrak g}$ is non--salvable . Also ${\mathfrak g}$ is not semi--simple, because its Killing form
$$ k=
\begin{bmatrix}
0 & 0 & 0 & 0 & 0  \\
0 & 0 & 0 & 0 & 0\\
0 & 0 & -2 & 0 & 0\\
0 & 0 & 0 & -2 & 0 \\
0 & 0 & 0 & 0 & -2
\end{bmatrix}$$
is degenerate. ${\mathfrak g}$ has a Levi--decomposition of the form ${\mathfrak g}={\mathfrak r}\oplus _{{\mathfrak h}}{\mathfrak h}$
 where ${\mathfrak r}=\langle {\bf X}_1,{\bf X}_2\rangle$ is the radical of ${\mathfrak g}$ and ${\mathfrak h}=\langle {\bf X}_3,{\bf X}_4,{\bf X}_5\rangle$
 is a semi--simple and non--solvable subalgebra of ${\mathfrak g}$.
%
\subsection{Classification of subalgebras}
Let $G$ be a Lie group with Lie algebra $\mathfrak{g}$.
There is an inner automorphism
$\tau_a\longrightarrow \tau \tau_a\tau^{-1}$ of the group $G$  for every arbitrary element $\tau\in G$. Every
automorphism of the group $G$ induces an automorphism of $\mathfrak{g}$.
The set of all these automorphism forms a Lie group called {\it the adjoint group $G^A$}.
For arbitrary infinitesimal generators ${\bf X}$ and ${\bf Y}$ in $\mathfrak{g}$, the linear mapping ${\rm
Ad}~{\bf X}({\bf Y}):{\bf Y}\longrightarrow[{\bf X},{\bf Y}]$ is an automorphism of $\mathfrak{g}$,
called {\it the inner derivation of $\mathfrak{g}$}. The set of all these
inner derivations equipped  by the Lie bracket $[{\rm Ad}{\bf X},{\rm Ad}{\bf Y}]={\rm Ad}[{\bf X},{\bf Y}]$ is a Lie
algebra $\mathfrak{g}^A$ called the {\it adjoint algebra of $\mathfrak{g}$}.
Two subalgebras in $\mathfrak{g}$ are {\it conjugate}
if there is a transformation of $G^A$ which takes one subalgebra
into the other. The collection of pairwise non-conjugate
$p$-dimensional subalgebras is called the {\it optimal system} of subalgebras
of order $p$ which is introduced by Ovsiannikov
\cite{Ovsiannikov}. Actually solving  the optimal system problem is to determine the conjugacy inequivalent
subalgebras with the property that any other subalgebra is
equivalent to a unique member of the list under some element of
the adjoint representation i.e. $\overline{\mathfrak{h}}\,{\rm Ad(\tau)}\,\mathfrak{h}$ for some $\tau$ in a given Lie group.

The adjoint action is given by the Lie series
\begin{eqnarray}
{\rm Ad}(\exp(q\,{\bf X}_i)){\bf X}_j={\bf X}_j-q\,[{\bf X}_i,{\bf X}_j]+\frac{q^2}{2}\,[{\bf X}_i,[{\bf X}_i,{\bf X}_j]]-\cdots,
\end{eqnarray}
where $q$ is a parameter and $i,j=1,\cdots,n$. Suppose
\begin{eqnarray}\label{eq:vectorfield}
{\bf X}=\sum_{i=1}^5a_i{\bf X}_i,
\end{eqnarray}
is an arbitrary member of the Lie algebra $\mathfrak{g}=\langle {\bf X}_1, \cdots, {\bf X}_5 \rangle$.
 Note that the elements of $\mathfrak{g}$ can be represented by vectors
$(a_1, \ldots, a_5)\in{\Bbb R}^5$ since
each of them can be written in the form (\ref{eq:vectorfield}) for some constants $a_1, \ldots, a_5$.
Hence, the adjoint action can be regarded as  a group of linear transformations of the vectors $(a_1, \ldots, a_5)$.
\begin{theorem}
An optimal system of one-dimension Lie subalgebras associated to the Lie point symmetries algebra of \eqref{Liepoint1} is generated by
\begin{gather}
\begin{array}{lcl}
&1)~{\bf X}_1+a_2{\bf X}_2+a_5{\bf X}_5 \qquad 2)~{\bf X}_1+a_2{\bf X}_2+a_3{\bf X}_3\qquad 3)~{\bf X}_1+a_2{\bf X}_2+a_4{\bf X}_4\\[3mm]
&~~~4)~{\bf X}_2+a_5{\bf X}_5\qquad 5)~{\bf X}_2+a_3{\bf X}_3 6)~{\bf X}_2+a_4{\bf X}_4\qquad 7)~{\bf X}_3 \qquad 8)~{\bf X}_4 \qquad 9)~{\bf X}_5
\end{array}
\end{gather}
\end{theorem}
\begin{proof}
The function $F_{i}^{s_{i}}:{\mathfrak g}\mapsto {\mathfrak g}$ defined by
${\bf X}\mapsto {\rm Ad}(\exp(s_i{\bf X}_i).{\bf X})$ is a linear map for $i=1,\dots,5$.
The matrices $M_{i}^{s_{i}}$ of $F_{i}^{s_{i}}$ with respect to basis $\{{\bf X}_1,\cdots ,{\bf X}_5\}$ are
$$M_1^{s_1}=
\begin{bmatrix}
1 & 0 & 0 & 0 & 0\\
0 & 1 & 0 & 0 & 0\\
0 & 0 & 1 & 0 & 0\\
0 & 0 & 0 & 1 & 0\\
0 & 0 & 0 & 0 & 1
\end{bmatrix}
\qquad
M_2^{S_2}=
\begin{bmatrix}
1 & 0 & 0 & 0 & 0\\
0 & 1 & 0 & 0 & 0\\
0 & 0 & 1 & 0 & 0\\
0 & 0 & 0 & 1 & 0\\
0 & 0 & 0 & 0 & 1
\end{bmatrix} \qquad
 M_3^{s_3}=\begin{bmatrix}
1 & 0 & 0 & 0 & 0\\
0 & 1 & 0 & 0 & 0\\
0 & 0 & 1 & 0 & 0\\
0 & 0 & 0 & \cos(s_3) & -\sin(s_3)\\
0 & 0 & 0 & \sin(s_3) &  \cos(s_3)
\end{bmatrix}$$\vspace*{5mm}
$$M_4^{s_4}=
\begin{bmatrix}
1 & 0 & 0 & 0 & 0\\
0 & 1 & 0 & 0 & 0\\
0 & 0 & \cos(s_4) & 0 & \sin(s_4)\\
0 & 0 & 0 & 1 & 0\\
0 & 0 & -\sin(s_4) & 0 & \cos(s_4)
\end{bmatrix}
\qquad M_5^{s_5}=
\begin{bmatrix}
1 & 0 & 0 & 0 & 0\\
0 & 1 & 0 & 0 & 0\\
0 & 0 & \cos(s_5) & -\sin(s_5) & 0\\
0 & 0 & \sin(s_5) &  \cos(s_5)  & 0\\
0 & 0 & 0 & 0 & 1
\end{bmatrix},$$
therefore
\begin{eqnarray*}
&F_5^{s_5}\circ F_4^{s_4}\circ F_3^{s_3} \circ F_2^{s_2}\circ F_1^{s_1}:{\bf X}\mapsto a_1{\bf X}_1+a_2{\bf X}_2+[(\cos(s_4)\cos(s_5))a_3+\\
&(\cos(s_5)\sin(s_3)\sin(s_4)+\cos(s_3)\sin(s_5))a_4+
(\sin(s_3)\sin(s_5)-\cos(s_3)\sin(s_4)\cos(s_5))a_5]{\bf X}_3+\\
&[(-\sin(s_5)\cos(s_4))a_3+
(\cos(s_3)\cos(s_5)-\sin(s_3)\sin(s_4)\sin(s_5))a_4+
(\cos(s_3)\sin(s_4)\sin(s_5)\\
&+\sin(s_3)\cos(s_5))a_5]{\bf X}_4+
[(-\sin(s_3)\cos(s_4))a_4+(\cos(s_3)\cos(s_4))a_5]{\bf X}_5.
\end{eqnarray*}
The one-dimensional optimal system of subalgebras associated to \eqref{Liepoint1} is determined as follows.
The operator ${\bf X}$ is simplified with the following cases:
\begin{itemize}
\item
If $a_1\neq 0$ we have three cases:

 1) The coefficients ${\bf X}_3,{\bf X}_4$ could be vanished by $F_3^{s_3},F_4^{s_4}$ by setting $s_3=\arctan(\frac{a_4}{a_5})$ and $s_4=-\arctan(\frac{a_3}{a_5})$ respectively.
 By scaling ${\bf X}$ we can assume that $a_1=1$ so ${\bf X}$ is reduced to case (1).

2) The coefficients ${\bf X}_4,{\bf X}_5$ could be vanished by $F_3^{s_3},F_4^{s_4}$ by setting $s_3=\arctan(\frac{a_4}{a_5}), s_4=\arctan(\frac{a_5}{a_3})$ respectively.
By scaling ${\bf X}$ we can assume that $a_1=1$ so ${\bf X}$ is reduced to case (2).

3) The  coefficients ${\bf X}_3,{\bf X}_5$ could be vanished by $F_3^{s_3},F_5^{s_5}$, by setting $s_3=-\arctan(\frac{a_5}{a_4}), s_5=\arctan(\frac{a_3}{a_4})$ respectively.
By scaling ${\bf X}$ we assume that $a_1=1$ so  ${\bf X}$ is reduced to case (3).
\item
Now If $a_1=0, a_2\neq 0$,we have three cases:

1) The coefficients ${\bf X}_3,{\bf X}_4$  would be vanished by $F_3^{s_3},F_4^{s_4}$ by setting $s_3=\arctan(\frac{a_4}{a_5}),s_4=\arctan(\frac{a_3}{a_5})$ respectively. By scaling ${\bf X}$, we can put $a_2=1$ and then ${\bf X}$ is reduced to case (4).

2) The coefficients ${\bf X}_4,{\bf X}_5$ could be vanished by
$F_3^{s_3},F_4^{s_4}$,by setting
$s_3=\arctan(\frac{a_4}{a_5}),s_4=\arctan(\frac{a_5}{a_3})$
respectively. By scaling ${\bf X}$ we assume that $a_2=1$ so ${\bf X}$ is reduced to case (5)

3) The coefficients ${\bf X}_3,{\bf X}_5$ could be vanished by
$F_3^{s_3},F_5^{s_5}$, by setting
$s_3=-\arctan(\frac{a_5}{a_4}),s_5=\arctan(\frac{a_3}{a_4})$ respectively. By scaling ${\bf X}$, we can put $a_2=1$ and then ${\bf X}$ is reduced to case (6).
\item
If $ a_1=0 , a_2=0 , a_3\neq 0$,  the coefficients ${\bf X}_4,{\bf X}_5$ could be vanished by $F_3^{s_3},F_4^{s_4}$, by setting
$s_3=\arctan(\frac{a_4}{a_5}),s_4=\arctan(\frac{a_5}{a_3})$
respectively.By scaling ${\bf X}$ we can put $a_3=1$ so ${\bf X}$ is reduced to case (7).
\item
If $ a_1=0 , a_2=0 , a_3=0 ,a_4\neq 0$, the coefficient ${\bf X}_5$ could be vanished by $F_3^{s_3}$,by setting $s_3=-\arctan(\frac{a_5}{a_4})$. By scaling ${\bf X}$, we can put $a_4=1$ and then ${\bf X}$ is reduced to case (8).
\item
If $ a_1=0, a_2=0, a_3=0, a_4=0, a_5\neq 0$, by scaling ${\bf X}$ we can put $a_5=1$ so it reduces to case (9).
\end{itemize}
\end{proof}
\section{Some illustrative examples}
\subsection{Symmetry analysis of the metric \eqref{eq:VB} for $M(t)=1$ and $Q(t)=t$}
In this case the associated Lagrangian of the metric \eqref{eq:VB}  is
\begin{equation} \label{lagra1-1}
\mathcal L =-(1-\frac{1}{r}+\frac{t}{r^2})\dot{t}^2-2\dot{t}\dot{r}+r^2(\dot {\theta}^2+\sin^2(\theta)\dot {\phi}^2).
\end{equation}
Suppose that the vector field ${\bf X}=\xi\;\partial_s+\eta^1\;\partial_t+\eta^2\;\partial_r+\eta^3\;\partial_\theta+\eta^4\partial_\varphi$
is the Noether symmetry operator of Lagrangian \eqref{lagra1-1}.
By solving the equation ${\bf X}^{[1]}{\mathcal L}+(D_s\xi){\mathcal L}=0$
 we can find the following Noether symmetries:
\begin{align}\label{nother1-1-1}
{\bf X}_1=\partial_s, \quad {\bf X}_2=\partial_\varphi,  \quad
{\bf X}_3=\cot(\theta)\sin(\varphi)\partial_\varphi-\cos(\varphi)\partial_\theta, \quad
{\bf X}_4=\cot(\theta)\cos(\varphi)\partial_\varphi+\sin(\varphi)\partial_\theta.
\end{align}
The Euler-Lagrange geodesic equations associated with the Lagrangian \eqref{lagra1-1} are
\begin{eqnarray}\label{eq:lagran1-1}
\left\{\begin{array}{lcl}
E_1:5\ddot{t}+\frac{2t-r}{4r^3}\ddot{t}^2+\frac{r}{2}\dot{\theta}^2+\frac{r\sin^2\theta}{2}\dot{\phi}^2=0\\ [2mm]
E_2:3\ddot{r}+\frac{2r^3+(r^2+t-r)(r-2t)}{8r^5}\dot{t}^2+\frac{r-2t}{2r^3}\dot{t}\dot{r}
+\frac{r-r^2-t}{4r}(\dot{\theta}^2+\sin^2\theta\dot{\phi}^2)=0\\ [2mm]
E_3:3\ddot{\theta}+\frac{2}{r}\dot{r}\dot{\theta}-\sin(\theta)\cos(\theta)\dot{\phi}^2=0\\ [2mm]
E_4:4\ddot{\phi}+\frac{2}{r}\dot{r}\dot{\phi}+2\cot(\theta)\dot{\theta}\dot{\phi}=0.
\end{array}\right.
\end{eqnarray}
Now suppose ${\bf X}$ is a Lie point symmetry generator of (\ref{lagra1-1}) thus ${\bf X}^{[2]}E_i\big|_{E_i=0}=0$ for $i=1,2,3,4$.
By solving this equation we obtain
\begin{eqnarray}
 \xi=c_1, \; \eta^1=0, \;
\eta^2=0\quad \eta^3=-c_3\cos(\varphi)+c_4\sin(\varphi),\;
\eta^4=c_2+\cot(\theta)(c_3\sin(\varphi)+c_4\cos(\varphi)).
\end{eqnarray}
Therefore Lie symmetry generators associated to the  metric \eqref{eq:VB} for $M(t)=1$ and $Q(t)=t$ are:
\begin{align}\label{liepoint1-1}
{\bf X}_1=\partial_s \quad {\bf X}_2=\partial_\varphi \quad
{\bf X}_3=\cot(\theta)\sin(\varphi)\partial_\varphi-\cos(\varphi)\partial_\theta \quad
{\bf X}_4=\cot(\theta)\cos(\varphi)\partial_\varphi+\sin(\phi)\partial_\theta,
\end{align}
and  commutator table of these symmetry generators is given in following table:
\begin{table}[ht]
\caption{commutator table (\ref{liepoint1-1})}
\centering
\begin{tabular}{c c c c c}
\hline\hline
  [~,~] & ${\bf X}_1$ & ${\bf X}_2$ & ${\bf X}_3$ & ${\bf X}_4$ \\
\hline
${\bf X}_1$ & 0 & 0 & 0 & 0 \\
\hline
${\bf X}_2$ & 0 & 0 & ${\bf X}_4$ & -${\bf X}_3$\\
\hline
${\bf X}_3$ & 0 & -${\bf X}_4$ & 0 & ${\bf X}_2$  \\
\hline
${\bf X}_4$ & 0 & ${\bf X}_3$ & $-{\bf X}_2$ & 0\\
  \hline
\end{tabular}
\end{table}

The algebra ${\mathfrak g}=\langle{\bf X}_1, {\bf X}_2,\cdots,{\bf X}_4\rangle$ is non-solvable because
 ${\mathfrak g}\supset {\mathfrak g^{(1)}}\neq 0$ where  ${\mathfrak g^{(1)}}=[{\mathfrak g} ,{\mathfrak g}]=\langle{\bf X}_2,{\bf X}_3,{\bf X}_4\rangle$.
 Also ${\mathfrak g}$ is not semi-simple, because its Killing form
$$ k=
\begin{bmatrix}
0 & 0 & 0 & 0  \\
0 & -2 & 0 & 0 \\
0 & 0 & -2 & 0 \\
0 & 0 & 0 & -2
\end{bmatrix}$$
is degenerate. ${\mathfrak g}$ has a Levi-decomposition of the form
 ${\mathfrak g}={\mathfrak r}\oplus _{{\mathfrak h}}{\mathfrak h}$ where ${\mathfrak r}=\langle {\bf X}_1\rangle$
 is the radical of ${\mathfrak g}$ and ${\mathfrak h}=\langle {\bf X}_2,{\bf X}_3,{\bf X}_4\rangle$ is a semi-simple and non-solvable subalgebra of ${\mathfrak g}$.\\
\subsection{Symmetry analysis of the metric \eqref{eq:VB} for $M(t)=t, Q(t)=t^2$}
The  Vaidya metric in  special case, $M(t) = t$, is known as the Papapetrou model \cite{Dwivedit}. Therefore in this case we consider
the Vaidya-Bonner metric  \eqref{eq:VB} for $M(t)=t ,Q(t)=t^2$.

The Lie algebra of Noether symmetries associated to metric \eqref{eq:VB} for $M(t)=t ,Q(t)=t^2$ is a 5-dimensional algebra spanned by the following vector fields:
\begin{eqnarray*}
&{\bf X}_1=s\;\partial_s+\frac{t}{2}\;\partial_t+\frac{r}{2}\;\partial_r \quad {\bf X}_2=\partial_s \quad
{\bf X}_3=\partial_\varphi \qquad {\bf X}_4=\cot(\theta)\sin(\varphi)\partial_\varphi-\cos(\varphi)\partial_\theta&\\
&{\bf X}_5=\cot(\theta)\cos(\varphi)\partial_\varphi+\sin(\varphi)\partial_\theta.&
\end{eqnarray*}
The Euler-Lagrange geodesic equations are
\begin{eqnarray}\label{eq:lagranu-u^2}
\left\{\begin{array}{lcl}
E_1:3\ddot{t}+(\frac{2t^2-rt}{4r^3})\dot{t}^2+\frac{r}{2}\dot{\theta}^2+\frac{r\sin^2(\theta)}{2}\dot{\phi}^2=0\\ [2mm]\label{eq:lagranu-u^2-1}
E_2:5\ddot{r}+\frac{2r^3(2t-r)+(r^2-rt+t^2)(rt-2t^2)}{8r^5}\dot{t}^2+\frac{tr-2t^2}{2r^3}\dot{r}\dot{t}
+\frac{r^2-rt+t^2}{4r}(\sin^2(\theta)\dot{\phi}^2-\dot{\theta}^2)=0\\ [2mm] \label{eq:lagrau-u^2-2}
E_3:3\ddot{\theta}+\frac{1}{r}\dot{r}\dot{\theta}-\sin(\theta)\cos(\theta)\dot{\phi}^2=0\\ [2mm]\label{eq:lagranu-u^2-3}
E_4:4\ddot{\phi}+\frac{2}{r}\dot{r}\dot{\phi}+2\cot(\theta)\dot{\theta}\dot{\phi}=0,
\end{array}\right.
\end{eqnarray}
by applying the symmetry criteria  we have
\begin{eqnarray*}
&\xi=c_1s+c_2 \quad \eta^1=c_1t, \quad \eta^2=c_1r, \quad
\eta^3=-c_4\cos(\varphi)+c_5\sin(\varphi)&\\
&\eta^4=c_3+\cot(\theta)(c_4\sin(\varphi)+c_5\cos(\varphi)).&
\end{eqnarray*}
Thus the Lie point symmetries are as follows:
\begin{eqnarray}\nonumber
&{\bf X}_1=s\;\partial_s+t\;\partial_u+r\;\partial_r \quad {\bf X}_2=\partial_s, \quad
{\bf X}_3=\partial_\varphi \qquad {\bf X}_4=\cot(\theta)\sin(\varphi)\partial_\varphi-\cos(\varphi)\partial_\theta&\\\label{eq:12545}
&{\bf X}_5=\cot(\theta)\cos(\varphi)\partial_\varphi+\sin(\varphi)\partial_\theta.&
\end{eqnarray}

The commutator table of symmetry generators is given in following table:
\begin{table}[ht]
\caption{commutator table (\ref{eq:12545})}
\centering
\begin{tabular}{c c c c c c}
\hline\hline
[~,~] & ${\bf X}_1$ & ${\bf X}_2$ & ${\bf X}_3$ & ${\bf X}_4$ & ${\bf X}_5$\\
\hline
${\bf X}_1$ & 0 & $-{\bf X}_2$ & 0 & 0 & 0\\
\hline
${\bf X}_2$ & ${\bf X}_2$ & 0 & 0 & 0 & 0\\
\hline
${\bf X}_3$ & 0 & 0 & 0 & ${\bf X}_5$ & $-{\bf X}_4$ \\
\hline
${\bf X}_4$ & 0 & 0 & $-{\bf X}_5$ & 0 & ${\bf X}_3$ \\
\hline
${\bf X}_5$ & 0 & 0 & ${\bf X}_4$ & $-{\bf X}_3$ & 0\\
\hline
\end{tabular}
\end{table}
${\mathfrak g}$ is non-solvable because we have ${\mathfrak g^{(1)}}=[{\mathfrak g} ,{\mathfrak g}]=$ $\langle {\bf X}_2,{\bf X}_3,{\bf X}_4,{\bf X}_5\rangle$ , ${\mathfrak g^{(2)}}=[{\mathfrak g^{(1)}} ,{\mathfrak g^{(1)}}]=$ $\langle {\bf X}_3,{\bf X}_4,{\bf X}_5\rangle$,  ${\mathfrak g^{(3)}}=[{\mathfrak g^{(2)}} ,{\mathfrak g^{(2)}}]=$ $\langle {\bf X}_3,{\bf X}_4,{\bf X}_5\rangle$, thus we have the following chain of ideals ${\mathfrak g}\supset {\mathfrak g^{(1)}}\supset {\mathfrak g^{(2)}}\dots \neq 0$ which shows that ${\mathfrak g}$ is non-salvable . Also ${\mathfrak g}$ is not semi-simple, because its Killing form
$$ k=
\begin{bmatrix}
1 & 0 & 0 & 0 & 0\\
0 & 0 & 0 & 0 & 0\\
0 & 0 & -2 & 0 & 0\\
0 & 0 & 0 & -2 & 0\\
0 & 0 & 0 & 0 & -2
\end{bmatrix}$$
is degenerate.
${\mathfrak g}$ has a Levi decomposition of the form ${\mathfrak g}={\mathfrak r}\oplus _{{\mathfrak h}}{\mathfrak h}$ where ${\mathfrak r}=\langle {\bf X}_1,{\bf X}_2\rangle$ is the radical of ${\mathfrak g}$ and ${\mathfrak h}=\langle {\bf X}_3,{\bf X}_4,{\bf X}_5\rangle$ is a semi-simple and non-solvable subalgebra of ${\mathfrak g}$.
\section*{Acknowledgement}
The second  author (Rohollah Bakhshandeh-Chamazkoti) acknowledges the funding support
of Babol Noshirvani University of Technology under Grant No. BNUT/391024/99.
%

\bibliographystyle{amsplain}

\begin{thebibliography}{10}
%
\bibitem{Bakhshandeh1}
{R}. {Bakhshandeh-Chamazkoti},
\newblock Symmetry analysis of the charged squashed Kaluza-Klein black hole metric,
\newblock {\em Mathematical  Methods in the Applied Sciences}, 39,  3163-3172, (2016).
%
\bibitem{Bakhshandeh2}
{R}. {Bakhshandeh-Chamazkoti},
\newblock  Geometry of the curved traversable wormholes of (3 + 1)-dimensional spacetime metric, 
\newblock {\em International Journal of Geometric Methods in Modern Physics}, 14(4)  1750048, (2017).
%
\bibitem{Bokhari}
{H}. {Bokhari}, {A}. {H}. {Kara},  {A}. {R}. {Kashif}, and {F}. {Zaman},
\newblock Noether Symmetries Versus Killing Vectors and Isometries of Spacetimes,
\newblock {\em Int. J. Theor. Phys.} 45, 1063 (2006).
%
\bibitem{Kara}
{H}. {Bokhari}, and {A}. {H}. {Kara}, 
 \newblock Noether versus Killing symmetry of conformally flat Friedmann metric,
\newblock {\em General Relativity and Gravitation}, 39, 2053-2059, (2007).
%
\bibitem{Dwivedit}
{I}. {H}. {Dwivedit}, and {P}. {S}. {Joshi},
\newblock On the Nature of Naked Singularities in Vaidya Spacetimes, 
\newblock {\em Class. Quant. Grav}. 6, 1599 (1989).
%
\bibitem{Ibragimov}
{N}. {H}. {Ibragimov}, 
\newblock Elementary Lie group analysis and ordinary differential equations, 
\newblock {\em Chichester: Wiely},  1999.
%
\bibitem{Mahomed}
{D}. {N}. {Khan} {Marwat}, {A}. {H}. {Kara}, and {F}. {M}. {Mahomed},  
\newblock Symmetries, Conservation Laws and Multipliers
via Partial Lagrangians and Noether's Theorem for Classically Non-Variational Problems,
\newblock {\em Int. J. Theor. Phys}. 46,  3022--3029, (2007).
%
\bibitem{Niu}
{Zhen-Feng} {Niu} and {Wen-Biao} {Liu},
\newblock Hawking radiation and thermodynamics of a Vaidya-Bonner black hole,
\newblock {\em Research in Astron. Astrophys}. 10(1) 33-38, (2010).
%
\bibitem{Narain1}
{R}. {Narain},  and  {A}. {H}. {Kara}, 
\newblock The Noether Conservation Laws of Some Vaidiya Metrics,
\newblock {\em Int J Theor Phys}, 49, 260-269,  (2010) .
%
\bibitem{Narain2}
{R}. {Narain},  and  {A}. {H}. {Kara}, 
\newblock Invariance analysis and conservation laws of the wave equation on Vaidya manifolds, 
\newblock {\em Pramana - J. Phys.}, 77(3), 555-570, (2011).
%
\bibitem{Noe1}
 {E}. {Noether}
 \newblock Invariante variations probleme. 
 \newblock {\em Nachr Akad Wiss Gott Math Phys Kl} 1918;2:235-57 
 \newblock (English translation in Transp Theory Stat Phys
1(3):186-207, 1971).
%
\bibitem{Olv1}
{P}. {J}. {Olver}, 
\newblock Applications of Lie groups to differential equations, 
\newblock {\em New York: Springer}, (1986).
%
\bibitem{Ovsiannikov}
 {L}. {V}. {Ovsiannikov}, 
 \newblock Group Analysis of Differential Equations, 
 \newblock {\em Academic Press, New York}, (1982).
 %
 \bibitem{Tsamparlis}
{M}. {Tsamparlis}, and {A}. {Paliathanasis},
\newblock Lie and Noether symmetries of geodesic equations and collineations,
\newblock {\em Gen Relativ Gravit} 42, 2957-2980, (2010).
%
\bibitem{Paliathanasis}
{M}. {Tsamparlis}, and {A}. {Paliathanasis},
\newblock Lie symmetries of geodesic equations and projective
collineations, 
\newblock {\em Nonlinear Dyn}  6(2) 203-214, (2010).
%
\end{thebibliography}

\end{document}